\documentclass{amsart}
\usepackage[margin=1.5in]{geometry}

\usepackage[all]{xy}
\usepackage{tikz-cd}
\usepackage{enumerate}
\usepackage{amsfonts}
\usepackage{amssymb, amscd, amsmath}
\usepackage{graphicx}
\usepackage{mathtools}
\usepackage{xcolor}
\usepackage{float}
\usepackage[neverdecrease]{paralist}
\usepackage{url}

\DeclareMathOperator{\Sup}{Sup}
\DeclareMathOperator{\conv}{conv}

\newtheorem{teo}{Theorem}[section]
\newtheorem{cor}[teo]{Corollary}

\newtheorem{lema}[teo]{Lemma}

\theoremstyle{definition}
\newtheorem{defn}[teo]{Definition}

\title{Peabodies of Constant Width}

\author{Isaac Arelio}
\author[L. Montejano]{Luis Montejano   }
\author{Deborah Oliveros}

\begin{document}
\maketitle

\section{Introduction}

Bodies of constant width and their properties have been known for centuries. Leonard Euler, for example, studied them in the eighteenth century under the name of orbiforms.  
They have received considerable attention in popular mathematics, in such contexts as videos, surveys, devices, and art, among others.
There is a broad, diverse body of knowledge on bodies of constant width supported by an extensive and sophisticated theoretical framework. See, for instance, the book \emph{Bodies of Constant Width: An Introduction to Convex Geometry With Applications} by Birkh\"auser, 2019 \cite{MMO}. 

It is well known that there is a non-constructive procedure to complete a set 
to a body of constant width of the same diameter, but besides the two Meissner solids, the obvious constant width bodies
of revolution, and the Meissner polyhedra, there are only a few tangible examples in the literature of constant width bodies that have a concrete finite procedure of construction. 
The purpose of this paper is to describe  a new $3$-dimensional family of bodies of constant width that we have called peabodies, obtained from the Reuleaux tetrahedron by replacing a small neighborhood of all six edges with sections of an envelope of spheres. This family contains, in particular, the two Meissner solids and a body with tetrahedral symmetry that we have called Robert's body, described by Patrick Roberts (see Figure \ref{patric})  (we encourage the reader to watch the animation presented at the beginning of \cite{RP} as well). 

\begin{figure}[h]
\begin{center}
\includegraphics[scale=.3]{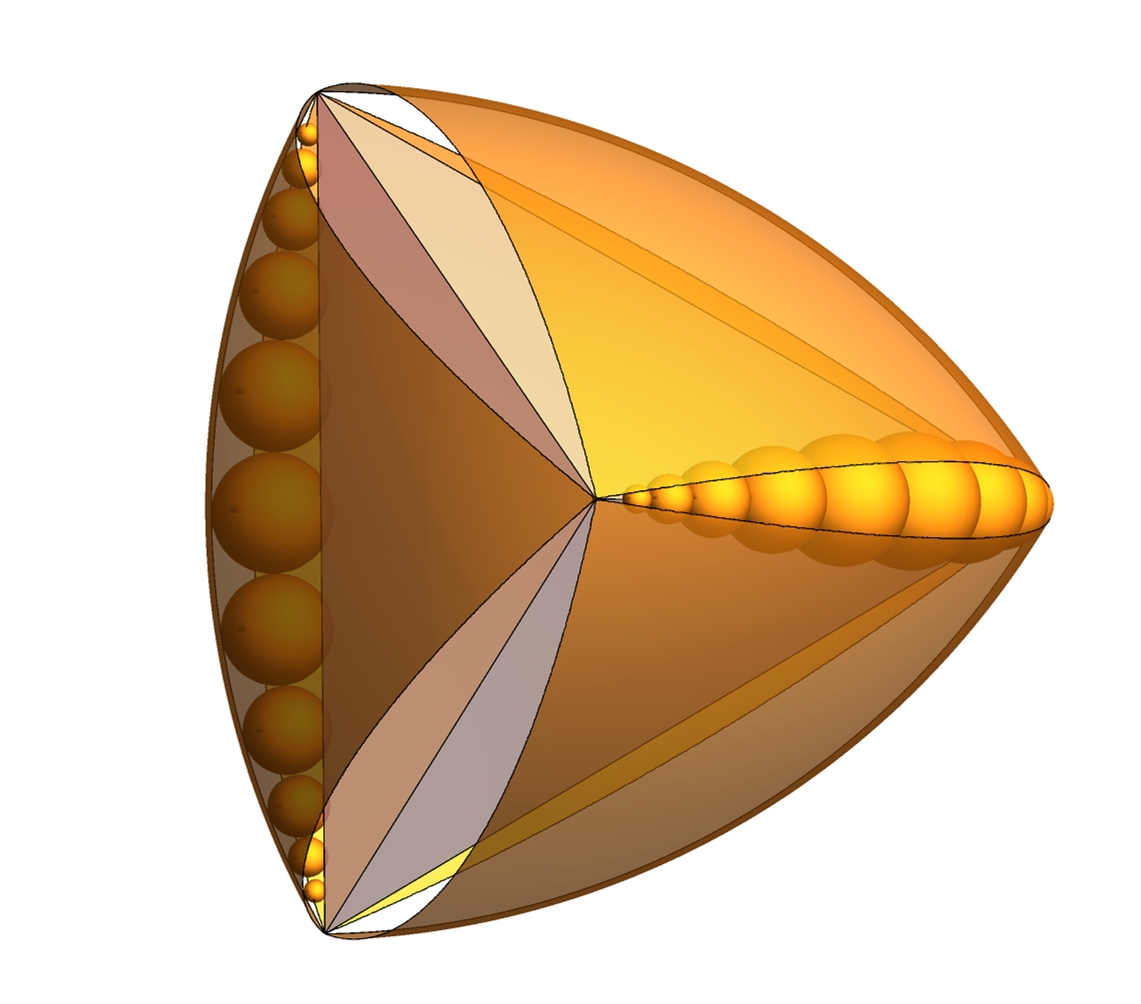}
\caption{Robert's body of constant width}\label{patric}
\end{center}
\end{figure}

It was a very pleasant surprise to discover that behind the construction of this family lies the classical notion of confocal quadrics discussed, for example, by Hilbert in his famous book \emph{Geometry and Imagination} \cite{H}.  This notion has been used as early as Dupin in the nineteenth century to build surfaces that are envelopes of spheres with surprisingly interesting properties; see \cite{D}. 

\begin{figure}[h]
\begin{center}
\includegraphics[scale=.42]{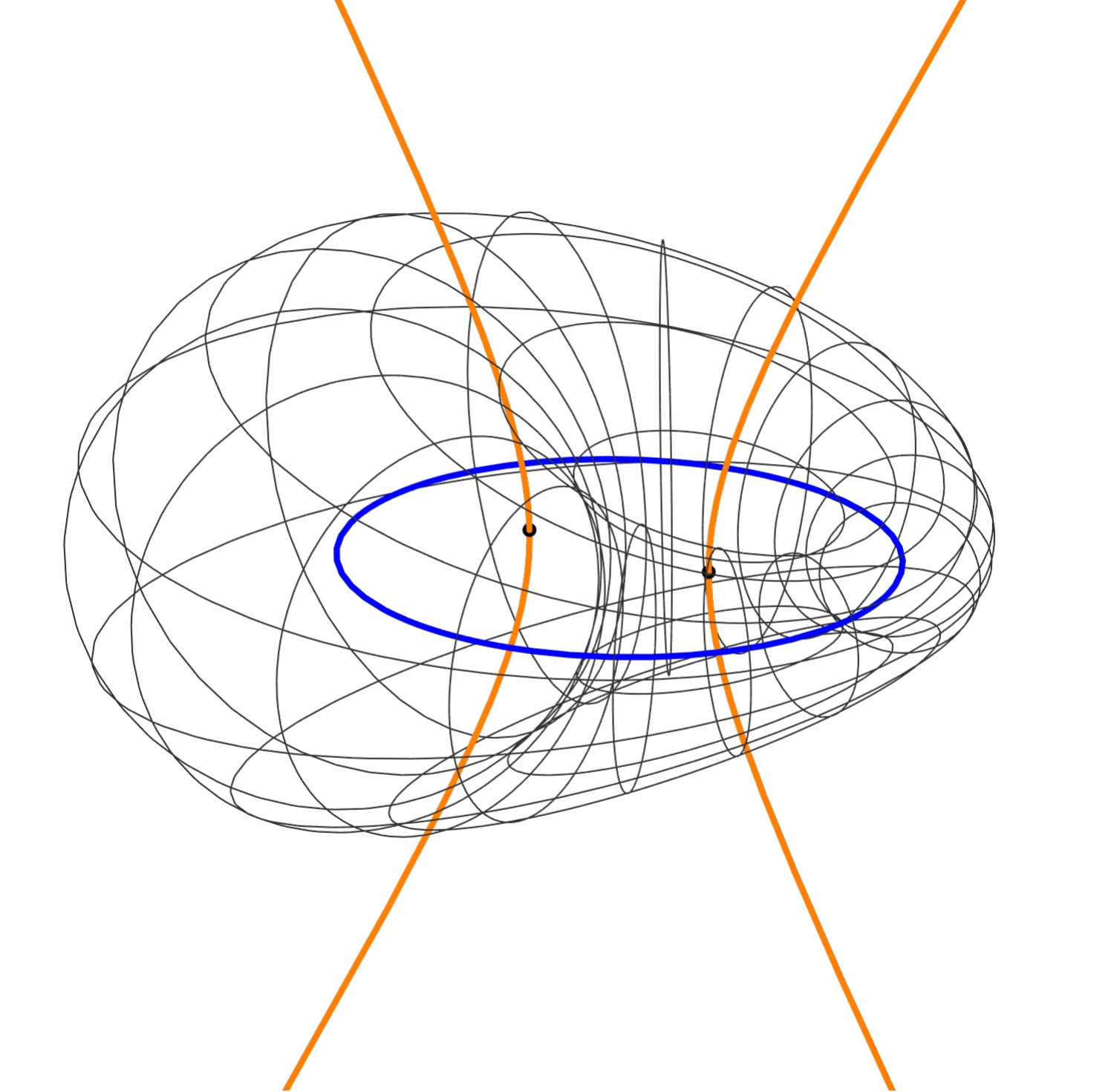}
\caption{Dupin cyclide}\label{}
\end{center}
\end{figure}

In Section 2, we study confocal quadrics and prove that the distances of an alternating sequence of four points in two confocal quadrics always satisfies a simple equation; a result that is interesting in its own right. In Sections 3 and 4 we construct this new family of $3$-dimensional bodies and show that they are bodies of constant width. Next, in Section 5 we analyze the particular case of Robert's body, showing that it has tetrahedral symmetry and its boundary is smooth except for the $4$ vertices of the tetrahedron, in which it has vertex singularities. Moreover, we observe that this body is not the Minkowski sum of the two Meissner bodies of constant width by showing that they differ in one of their sections.
We also note that Robert's body can not be the extremum body for the Blaschke-Lebesgue conjecture about the minimum volume among all 3-dimensional bodies of constant width. Furthermore, we show that there is a continuous deformation along the collection of peabodies of constant width from the most symmetric Robert's body to the classic Meissner bodies.


Finally, in Section 6 we point out that the construction of bodies of constant width obtained from ball polyhedra whose singularities are self dual-graphs,
defined in \cite{MR} by Montejano and Roldan, can be adapted by replacing the singularities of these ball polyhedra with sections of an envelope of spheres to achieve constant width without changing the group of symmetries of the ball polyhedra.

\section{Confocal quadrics}

Let us consider two quadrics lying in orthogonal planes whose axes are the same.  See Figure \ref{cocon}. We say that these two conics are confocal if the focus of each of them lies on the other. See Section 2.4 of the  book \emph{Geometry and the Imagination} by Hilbert and Cohen-Vossen \cite{H}.

\begin{figure}[h]
\begin{center}
\includegraphics[scale=.5]{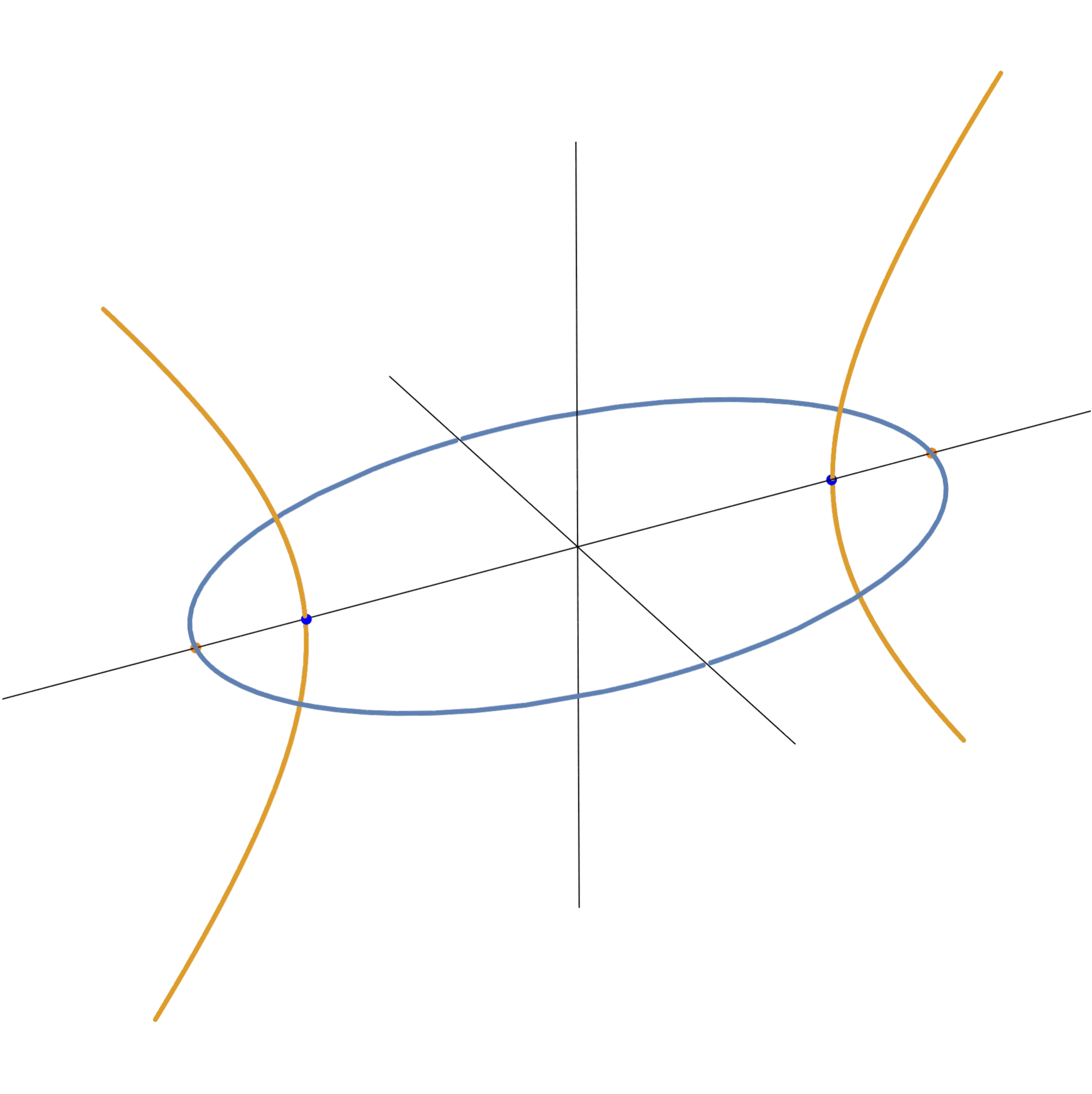}
\caption{Confocal quadrics }\label{cocon}
\end{center}
\end{figure}


In the standard case, when one of the two confocal quadrics is the ellipse and the other  a hyperbola, Equation \eqref{elipse-hyperbola} gives us its planar equations. 

\begin{equation}\label{elipse-hyperbola}
\frac{z^2}{a^2}+\frac{x^2}{b^2}=1 \quad\quad \text{ and } \quad\quad \frac{z^2}{a^2-b^2}-\frac{y^2}{b^2}=1.	
\end{equation}

Similarly, in the case when the two confocal quadrics are parabolas and the origen is the midpoint of their focus, 
the planar equations are given in \eqref{parabolapor2}.  

\begin{equation}\label{parabolapor2}
z=\frac{x^2}{4a}-\frac{a}{2} \quad\quad \text{ and }\quad\quad z=-\frac{y^2}{4a}+\frac{a}{2}.
\end{equation}

The following theorem is interesting in its own right and, as we will see in Section \ref{peapods}, it is relevant to our construction of peabodies of constant width. Throughout the paper we will abuse the notation by denoting by $ab$ both the interval with extremes at the points $a$ and $b$, as well as its length.

\begin{teo}\label{unoyuno} Let $C_1$ and $C_2$ be two confocal quadrics. Suppose $a_1\in C_1$,  $a_2\in C_2$, $a_3\in C_1$ and $a_4\in C_2$. 
\begin{enumerate}[a)]
\item If $C_1$ is connected and $a_2, a_4$ lie in the same component of $C_2$, then
\[a_1a_2+a_3a_4=a_2a_3+a_1a_4.\]
\item If $C_1$ is connected and $a_2$ and $a_4$ lie in different components of $C_2$, then
\[a_1a_2+a_1a_4=a_2a_3+a_3a_4.\]
\end{enumerate}
\end{teo}

\begin{proof}
\emph{ The parabolic case}. 
Consider the most general case of two confocal parabolas with focus at $\pm\frac{a}{2}$, parametrized by
\[C_1\colon \Big(2 a t, 0, a t^2 - \frac{a}{2}\Big) \quad\quad \text{and} \quad\quad C_2\colon \Big(0, -2 a s, - a s^2+ \frac{a}{2} \Big),\] for $t, s\in\mathbb{R}$.

If we take any point $a_1=(2 a t_1, 0, a t_1^2 - \frac{a}{2})\in C_1$ and any point $a_2=(0, -2 a s_2, - a s_2^2 + \frac{a}{2})\in C_2$, then \vspace{-8pt}
\[a_1a_2=|(2 a t_1, 2 a s_2), -a + a s_2^2 + a t_1^2)|= \sqrt{ a^2 (1 + t_1^2 + s_2^2)^2}=a + at_1^2 + as_2^2.\]

If we now take any point $a_3=(2 a t_3, 0, a t_3^2 - \frac{a}{2})\in C_1$ and any point $a_4=(0, -2 a s_4,  -a s_4^2 + \frac{a}{2})\in C_2$, then \vspace{-8pt}
\begin{align*}
a_2a_3&= a + as_2^2 + at_3^2, \\
a_3a_4&= a + at_3^2 + as_4^2, \\
a_1a_4&= a + at_1^2 + as_4^2.
\end{align*}
Consequently,
\[a_1a_2+a_3a_4=a_2a_3+a_1a_4\]
as we wished.  

\medskip
\noindent \emph {The elliptic-hyperbolic case a)}.
 Let $C_1=E$ be the ellipse in the $xz$-plane parametrized by $(b \cos t, 0, a \sin t)$, $0<t\leq\pi$, where $a$ is the major axis and $b$ the minor axis. Let $C_2=H$ be the hyperbola in the $yz$-plane parametrized by $(0, a \tan s, -\sqrt{a^2 - b^2}\sec s)$, $\dfrac{\pi}{2}<s<\dfrac{3\pi}{2}$. Note that the ellipse 
$E$ has a focus at $(0,0,\pm \sqrt{a^2-b^2})$ and the hyperbola $H$ has a focus at $(0,0,\pm a)$.

Let $a_1=(b \cos t_1, 0, a \sin t_1)\in E$ and $a_2=(0, a \tan s _2, -\sqrt{a^2 - b^2}\sec s_2)\in H$.  
First, let us consider the square of the distance from $a_1$ to $a_2$:
\[\left|a_1-a_2\right|^2=b^2 \cos^2 t_1 + (a^2 - b^2) \sec ^2 s_2 + 
 2 a \sqrt{a^2 - b^2} \sec s_2 \sin t_1 + a^2 \sin^2 t_1 + b^2 \tan^2 s_2.
\]
Using the identities $\cos^2 t=1-\sin^2 t$ and $\tan^2 s=\sec^2 s -1$, we get
\begin{align}
|&a_1-a_2|^2=a^2 \sec^2 s_2 + 2 a \sqrt{a^2 - b^2} \sec s_2 \sin t_1 + (a^2 - b^2) \sin^2 t_1, \notag \\
|&a_1-a_2|^2=(a\sec s_2+\sqrt{a^2 - b^2} \sin t_1)^2, \notag \\
&a_1a_2=|(a_1-a_2|= a\sec s_2+\sqrt{a^2 - b^2} \sin t_1.\label{eq1}
\end{align}


Now let $a_3=(b \cos t_3, 0, a \sin t_3)\in E$ and $a_4=(0, a \tan s _4, -\sqrt{a^2 - b^2}\sec s_4)\in H$, where $a_2, a_4$ lie in the same component of $H$. By \eqref{eq1}, 
\[ a_3a_4= a\sec s_4+\sqrt{a^2 - b^2} \sin t_3.\]
Similarly, by \eqref{eq1},
\begin{align*}
a_3a_2&= a\sec s_2+\sqrt{a^2 - b^2} \sin t_3, \text{ and } \\
a_1a_4&= a\sec s_4+\sqrt{a^2 - b^2} \sin t_1. 
\end{align*}

Consequently, $a_1a_2+a_3a_4=a_2a_3+a_1a_4,$ as we wished.

\medskip\noindent\emph{The elliptic-hyperbolic case b)}.
This is the case in which $a_4$ is in a different component of $a_2$. Let us temporarily fix $a_1,a_3 \in E=C_1$ and let $H_1$ and $H_2$ be the two components of the hyperbola $H=C_2$ in such a way that $a_2\in H_1$ and $a_4\in H_2$.

Note that by a), \vspace{-6pt}
\[a_1a_2-a_2a_3=k, \quad  \mbox{ for every } a_2\in H_1, \]
where $k$ is a constant, and similarly, \vspace{-2pt}
\[a_1a_4-a_4a_3=\lambda, \quad  \mbox{ for every } a_4\in H_2, \]
where $\lambda$ is a constant.

Let us prove that $k=-\lambda$. For this purpose, let $s_i\in H_i$ be the foci of the ellipse $E$.
By the above, $a_1s_1-s_1a_3=k$ and $a_1s_2-s_2a_3=\lambda$. Since $s_1,s_2$ are the foci of $E$ and $a_1,a_3\in E$, we have that $k=-\lambda$.  But now $a_1a_2-a_2a_3=-(a_1a_4-a_4a_3)$,  which implies that $a_1a_2+a_1a_4=a_2a_3+a_3a_4$, as we wished. \end{proof}

Theorem \ref{unoyuno} gives geometric sense to the intuitive idea that for two confocal quadrics, each  of them is a focal curve of the other.  Indeed,  Theorem \ref{unoyuno} b) simultaneously summarizes both characteristic properties of ellipses and hyperbolas.


\section{Confocal Pea Pod Devices}\label{peapods}

\begin{defn}\label{defpeapod}
A \emph{frame of the pea pod}  $\Sigma$ in $\mathbb{R}^3$ consists of two circles in a plane $H$ whose centers lie on an \emph{axis} line $L$ such that their intersection contains a chord $I$ called the \emph{longitudinal beam}, perpendicular to $L$. Let $m=L\cap I$. We will say that a \emph{pea pod frame} is an \emph{elliptic frame} when $m$ is not between the centers of the circles, is an \emph{hyperbolic  frame} when $m$ is between the centers of the circles, and is a \emph{parabolic frame} when one of the circles is the line through $I$. 	
\end{defn}

Next, we will construct a set of elements that will make up what we call a pea pod device, which is crucial to the construction of peabodies and Theorem \ref{teomain}.

\subsection{Elliptic Pea Pod Device}

Consider an elliptic frame of a pea pod $\Sigma_e$. By definition, the point $m$ is not between the centers of the circles. Call the circle whose center $c_e$ is closer to $m$ the principal circle of the elliptic frame. The other circle is called the secondary circle.  

Consider the collection of all disks contained in the symmetric difference between the principal and secondary circles and tangent to both of them. This collection of disks is known as a Steiner chain (see for example \cite[pp.~51-54]{O}).  Observe that their centers lie in an ellipse contained in $H$ with the focus at the principal and secondary centers. We will call this curve \emph{the elliptic pea string} of $\Sigma_e$ and denote it by $C_e$; see Figure \ref{eppd}. 
 
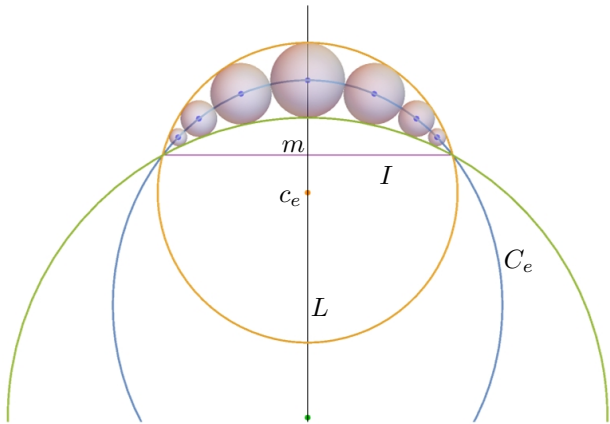
\begin{figure}[h]
\begin{center}
\includegraphics[scale=.5]{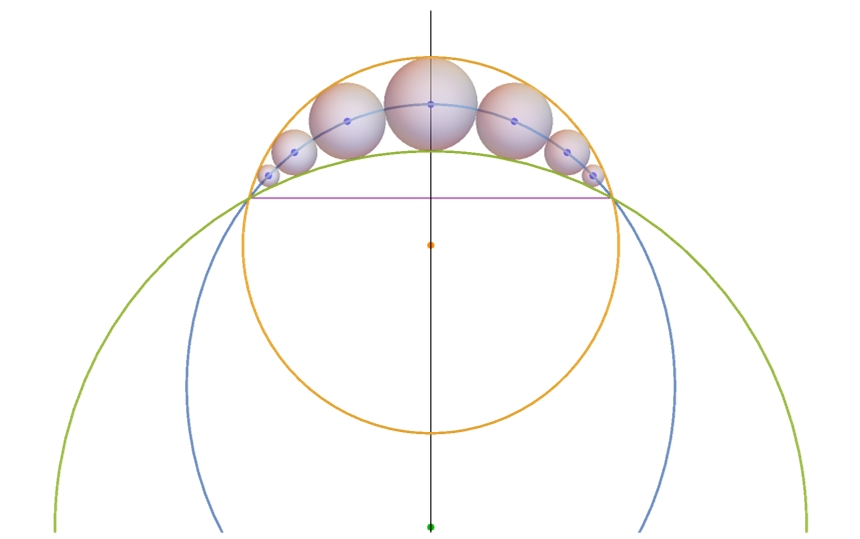}\put(-140,85){$c_e$}\put(-55,60){$C_e$}\put(-102,92){$I$}\put(-139,104){$m$}\put(-128,42){$L$}
\caption{Elliptic pea pod device $\Sigma_e$. For clarity the figure shows only few spheres (``peas'') but the reader may imagine one sphere centered at each point of the curve.}\label{eppd}
\end{center}
\end{figure}

Next, for every one of the disks in the Steiner chain and contained in the interior of the principal circle, consider the corresponding $3$-dimensional ball centered at the center of the disk and with the same radius. The collection of all these balls is called the \emph{elliptic pea pod} devise of $\Sigma_e$. The ball of the pea pod whose center lies in $L$ is called the \emph{bulb}. 

\bigskip

\subsection{Hyperbolic Pea Pod}
Now consider a hyperbolic frame $\Sigma_h$, and recall that in this case the point $m$ is between the centers of the circles. Here we will call the smaller circle of the device with center at $c_h$ the principal circle of the frame. 
Consider the collection of disks contained in the intersection between the two circles of the frame such that they are tangent to both of them. It is easy to see that their centers lie along an hyperbola with its focus at the centers of the circles. We will call this curve \emph{the hyperbolic pea string} of $\Sigma_h$ and denote it by $C_h$.  Finally, as before, for every one of these disks, we consider 
$3$-dimensional balls with the same center and the same radius. The collection of all these balls is called the \emph{hyperbolic pea pod device} of $\Sigma_h$. As before, the ball of the pea pod whose center lies in $L$ is called the bulb. See Figure \ref{hppd}.

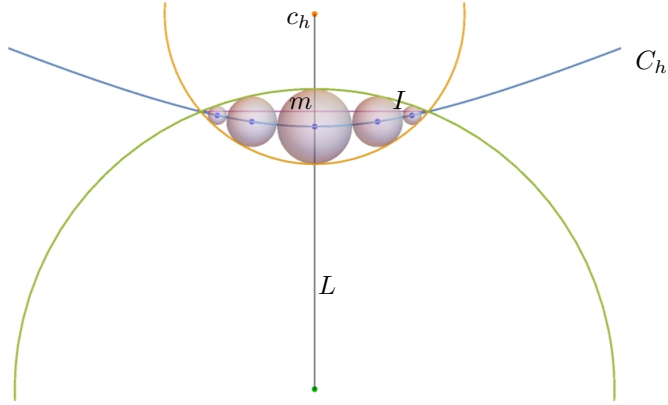
\begin{figure}[h]
\begin{center}
\includegraphics[scale=.5]{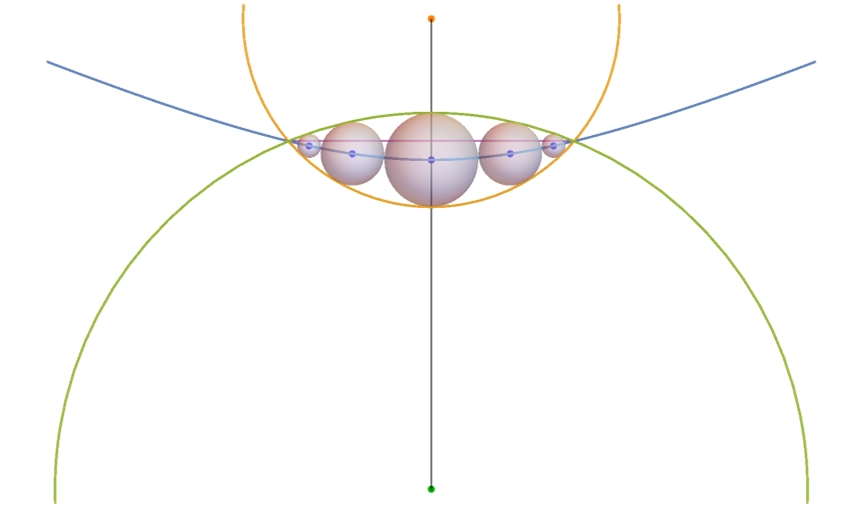}\put(-140,144){$c_h$}\put(-8,127){$C_h$}\put(-100,112){$I$}\put(-139,112){$m$}\put(-128,42){$L$}
\caption{Hyperbolic pea pod device $\Sigma_h$. }\label{hppd}
\end{center}
\end{figure}

\subsection{Parabolic Pea Pod}
In this case, $\Sigma_p$ consists of a frame where one of the circles is a line through $I$ and the other, called the principal circle of the frame,  has its center at $c_p$. Consider the collection of disks contained in the interior of the principal circle tangent to both the principal circle and the longitudinal beam. Consider only those disks that are in the half plane determined by $I$ opposite to $c_p$. It is easy to observe that in this case their centers lie along a parabola contained in $H$ with focus at $c_p$. Let us call this curve \emph{the parabolic pea string} of $\Sigma_p$ and denote it by $C_p$.  Finally, consider the $3$-dimensional balls centered at the center of all these disks and with the same radius. The collection of all these balls is called the \emph{pea pod} of $\Sigma_p$. The ball of the pea pod whose center lies in $L$ is called the bulb.
See Figure \ref{pppd}.

\begin{figure}[h]
\begin{center}
\includegraphics[scale=.51]{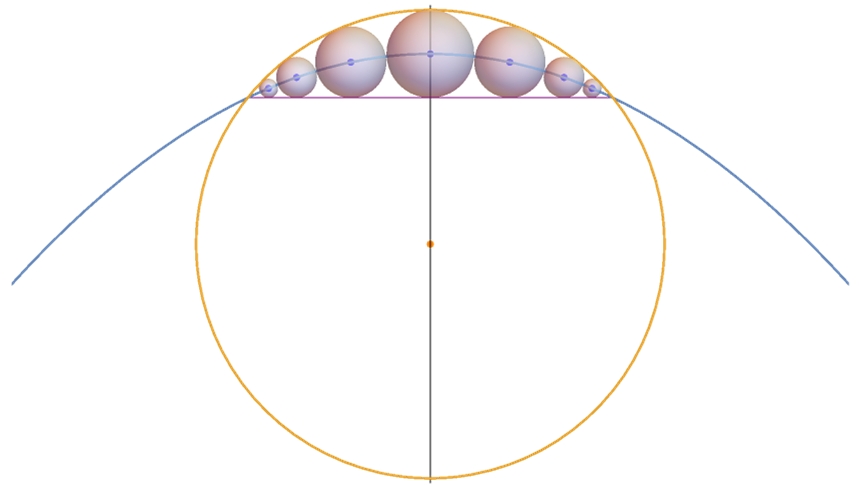}\put(-142,72){$c_p$}\put(-8,70){$C_p$}\put(-100,110){$I$}\put(-142,112){$m$}\put(-130,42){$L$}
\caption{Parabolic pea pod device $\Sigma_p$. }\label{pppd}
\end{center}
\end{figure}
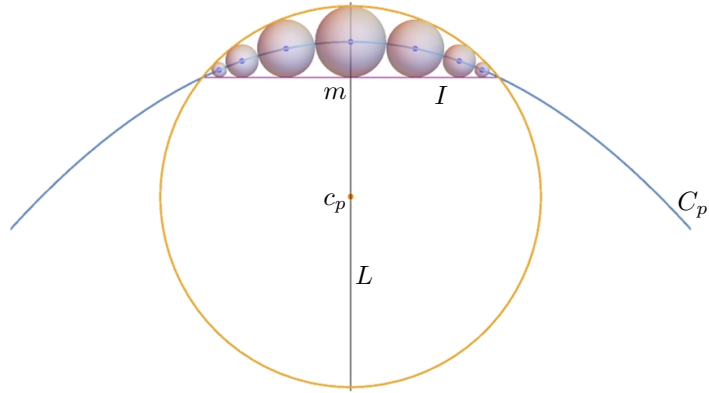

For all three cases and with a slight abuse of notation, $\Sigma_i$ will denote both the frame and the pea pod device.

\subsection{Confocal Pea Pods}


\begin{defn} Two pea pod devices $\Sigma_i$ and $\Sigma_j$ in $\mathbb{R}^3$, $i,j\in \{e,h,p\}$ are \emph{confocal} if they satisfy: 
\begin{enumerate}[1)]
\item $\Sigma_i$ and $\Sigma_j$ share the same axis $L$,
\item their corresponding planes are orthogonal,  and
\item their corresponding pea strings $C_i$ and $C_j$ are confocal quadrics.
\end{enumerate}
\end{defn}

Let $\Sigma_i$ and $\Sigma_j$ be two confocal pea pod devices, $i,j\in \{e,h,p\}$. Clearly,  if $i=e$ then $j=h$ (or vice versa) or if $i=p$ then $j=p$.

Let $\Sigma_i$ be a pea pod device $i\in \{e,h,p\}$ with principal circle at center $c_i$ and radius $r_i$. Denote by $E(\Sigma_i)$ the closed sub-arc of its quadric pea string consisting of those points that are centers of disks of the pea pod.  For every $x\in E(\Sigma_i)$, let us denote the $3$-dimensional ball of the pea pod centered at $x$ by $B(x)$ and its radius by $R(x)$.

The following theorem is our main result about confocal pea pods.

\begin{teo}\label{teodevice}
Let $\Sigma_i$ and $\Sigma_j$ be two confocal pea pod devices, $i,j\in \{e,h,p\}$. 
Then for every $x\in  E(\Sigma_i)$ and $y\in E(\Sigma_j)$, 
\[ xy+ R(x)+R(y)=r_i+r_j - c_ic_j.\]
\end{teo}

\begin{proof} First, recall that $E(\Sigma_i)$ and $E(\Sigma_j)$, 
are subarcs of the confocal pea quadric strings. Note that by construction
$ R (x)=r_i - c_ix$ and $R (y)=r_j - c_jy.$ 

Therefore, 
\[ xy+ R(x)+R(y)=\big(xy-c_ix-c_jy\big) + r_1+r_2.\]
Since $\{c_i, x, y, c_j\}$ are four alternate points in a pair of confocal quadrics and  $x$ and $c_j\in E(\Sigma_i)$ and $y$ and $c_i\in E(\Sigma_j)$, then each pair lies in the same component of the corresponding quadric. By Theorem  \ref{unoyuno} a), we have that
\[ xy+ R(x)+R(y)= r_i+r_j -c_ic_j,\]
which is  a constant independent of  $x$ and $y$. \end{proof}

\begin{cor}\label{cordevice}

Let $\Sigma_i$ and $\Sigma_j$ be two confocal pea pod devices, $i,j\in \{e,h,p\}$ with 
 $I=ab$ the longitudinal beam of $\Sigma_i$ and  $J=cd$ the longitudinal beam of $\Sigma_j$. Then
\[ac=ad=bc=bd=r_i+r_j -c_ic_j.\]
\end{cor}

\begin{defn} Let $\Sigma_i$ and $\Sigma_2$ be two confocal pea pod devices, $i,j\in \{e,h,p\}$. If in addition,
 the center $c_i$ of the principal circle of one pea pod device is the center of the bulb of the other,
then we will say that they are \emph{convex} confocal pea pod devices.  The notion of convexity for confocal pea pod devices and the following observation will be relevant for the proof of Lemma \ref{lemfam}.
\end{defn}

Observe that if $\Sigma_i$ and $\Sigma_j$ are convex confocal, $i,j\in \{e,j,p\}$,  with 
 $I=ab$ the longitudinal beam of $\Sigma_i$ and $J=cd$ the longitudinal beam of $\Sigma_j$, then  the tetrahedron $abcd$ satisfies that $E(\Sigma_i)\cap abcd=\{a,b\}$,  
$E(\Sigma_j)\cap abcd=\{c,d\}$, and the orthogonal projection along $L$ of $E(\Sigma_i)$ and $I$ coincide and the orthogonal projection along $L$ of $E(\Sigma_j)$ and $J$ coincide. (See Figure \ref{confo}.)

\begin{center}   
\begin{figure}[h]

\includegraphics[scale=.7]{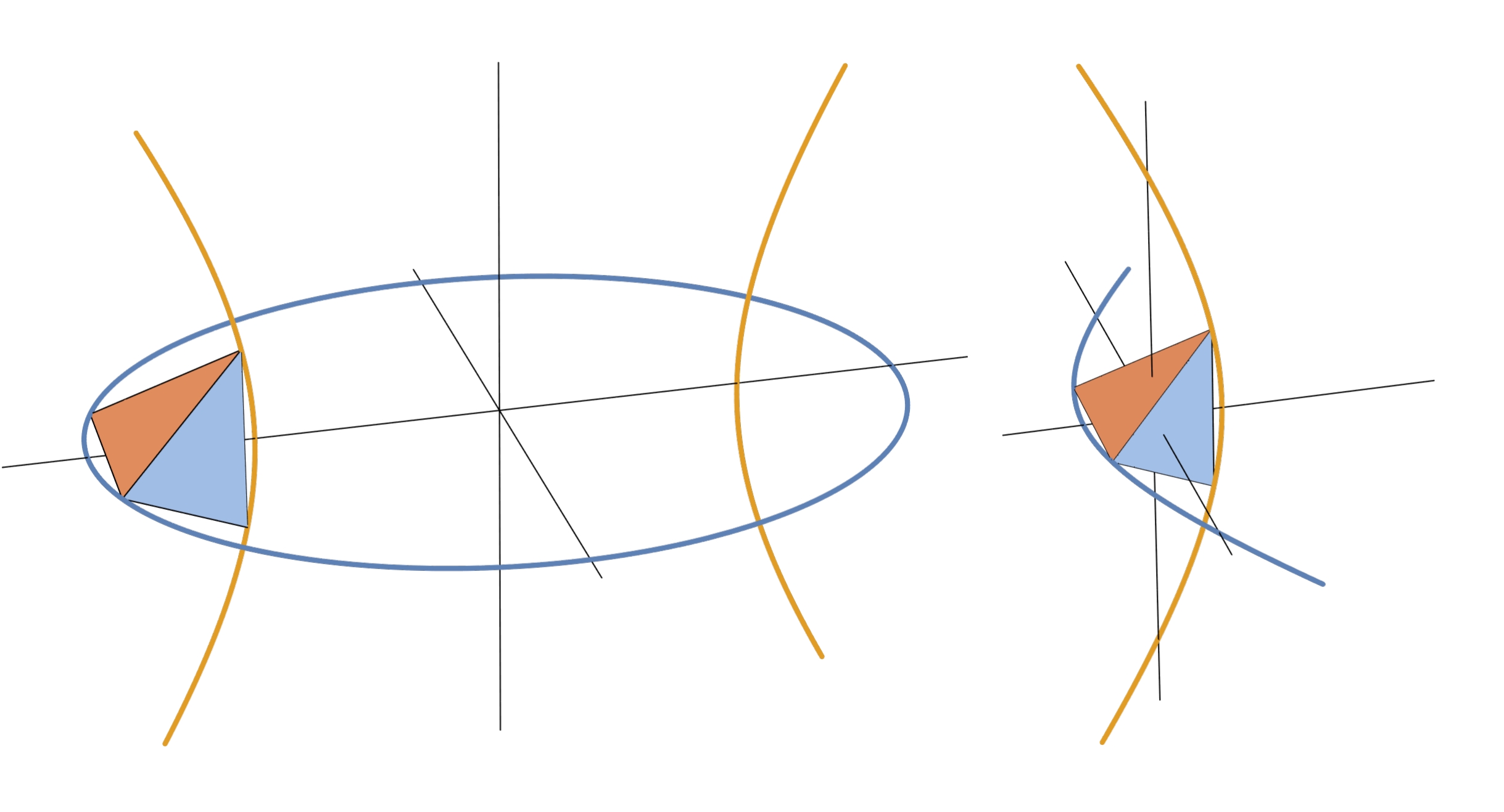}\put(-383,78){$a$}\put(-393,108){$b$}\put(-344,127){$c$}\put(-341,74){$d$}\put(-415,90){$E(\Sigma_e)$}\put(-340,100){$E(\Sigma_h)$}\put(-112,88){$a$}\put(-125,114){$b$}\put(-78,132){$c$}\put(-77,85){$d$}\put(-143,98){$E(\Sigma_i)$}\put(-75,105){$E(\Sigma_j)$}
\caption{Convex confocal pea pod devices}\label{confo}

\end{figure}
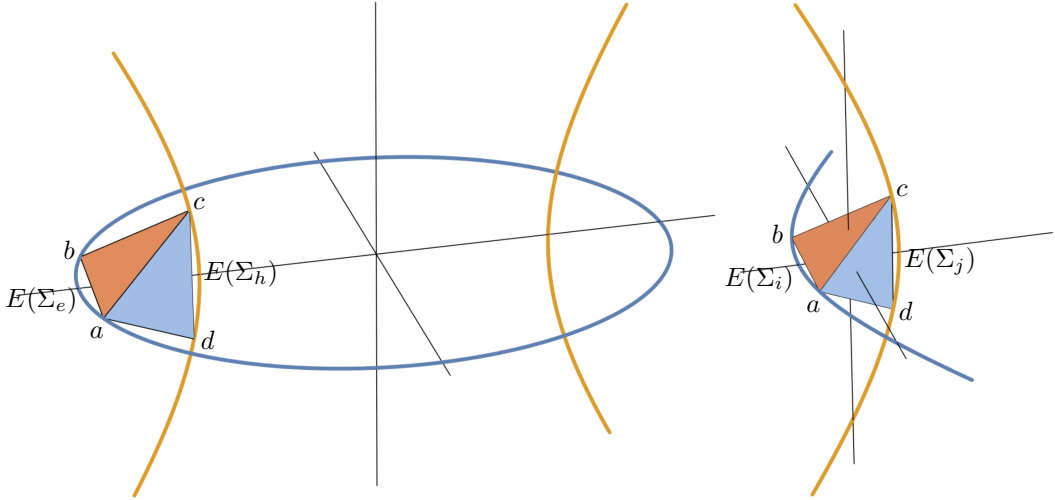
\end{center}

\subsection{The Wedge-Pod Surfaces of the Device}\label{wedges}

\begin{defn} Given a pea pod device $\Sigma$, the \emph{realization} of the pea pod of $\Sigma$, denoted by $|\Sigma|$, is defined as
\[|\Sigma|=\bigcup_{x\in E(\Sigma)} B(x).\] 

\end{defn}


As a consequence of Theorem \ref{teodevice} and Corollary \ref{cordevice}, we are able to measure the diameter between the realizations of two confocal pea pods. Recall that for two subsets $S, T$ of $\mathbb{R}^3$, the \emph{diameter} between $S$ and $T$ is defined by 
\[d(S,T)= \Sup\{ab\mid a\in S, b\in T\}.\]

\begin{lema}\label{lemS} 
Let $\Sigma_1$ and $\Sigma_2$ be two confocal pea pod devices. Then
\[d(\left|\Sigma_1\right|, \left|\Sigma_2\right|)\leq r_1+r_2 -c_1c_2.\]
\end{lema}

\begin{proof} Let $w_i\in |\Sigma_i|$, $i=1,2.$ Then, $w_i\in B( x_i),$  
for some point $x_i$ in the subarc $E(\Sigma_i)$ of its pea quadric string, $i=1,2.$  Therefore, by the triangle inequality, 
$w_1w_2\leq x_1x_2+ R( x_1)+R( x_2).$
By Theorem \ref{teodevice},  $w_1w_2\leq r_1+r_2 -c_1c_2$.\end{proof}

 Let $\Sigma_1$ and $\Sigma_2$ be two confocal pea pod devices and as before 
denote by $I_1=ab$ the longitudinal beam of $\Sigma_1$ and by $I_2=cd$ the longitudinal beam of $\Sigma_2$.
For every $x$ in the sub-arc $E(\Sigma_1)$ and $y$ in the sub-arc $E(\Sigma_2)$, let  $L(x,y)$ 
be the line through $x$ and $y$. (See Figure \ref{isaac}.)

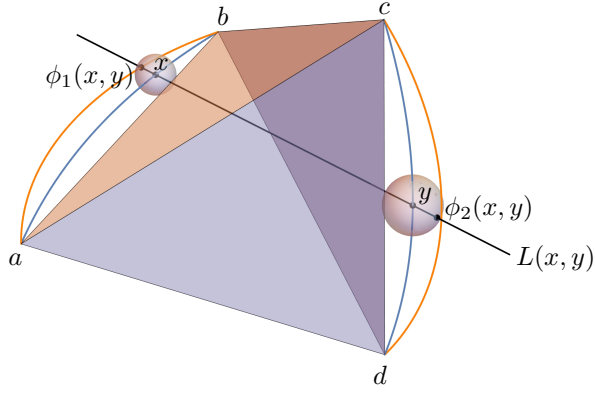
\begin{figure}[h]
\begin{center}
\includegraphics[scale=.35]{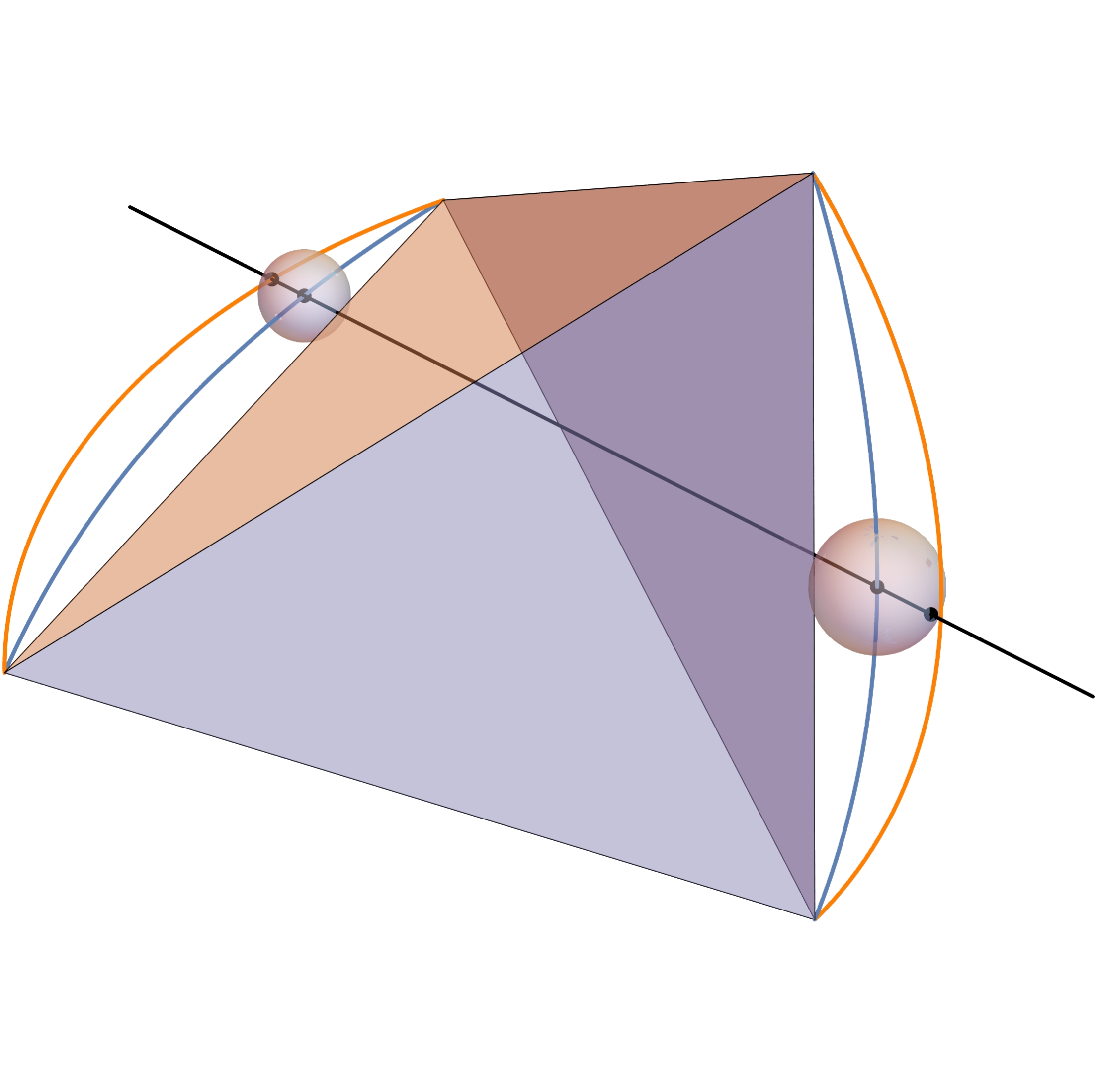}\put(2,63){$L(x,y)$}\put(-25,82){$\phi_2(x,y)$}\put(-35,88){$y$}\put(-135,137){$x$}\put(-176,132){$\phi_1(x,y)$}\put(-190,63){$a$}\put(-52,18){$d$}\put(-50,158){$c$}\put(-111,154){$b$}
\caption{The line $L(x,y)$ }\label{isaac}
\end{center}
\end{figure}

  Denote by $\phi_1(x,y)$ the point in $L(x,y)$ at a distance $R(x)$ from $x$, with $x$ between $\phi_1(x,y)$ and $y$ and let 
  $\phi_2(x,y)$ be the point in $L(x,y)$ at a distance $R(y)$ from $y$, with $y$ between $x$ and $\phi_2(x,y)$.  Note that 
  $\phi_1(x,y)\in |\Sigma_1|$, $\phi_2(x,y)\in |\Sigma_2|$ and   
$\phi_1(x,y)\phi_2(x,y)= r_1+r_2 -c_1c_2$. Therefore, $d(|\Sigma_1|, |\Sigma_2|)=r_1+r_2 -c_1c_2.$

Moreover, the interval $\phi_1(x,y)\phi_2(x,y)$ is a binormal of the convex hull  $\conv(|\Sigma_1|\cup|\Sigma_2|)$.  
Indeed, if $x\not= \{a, b\}$, then there is a unique support plane of  $\conv(|\Sigma_1|\cup|\Sigma_2|)$ at $\phi_1(x,y)$ orthogonal to $\phi_1(x,y)\phi_2(x,y)$ because this is the unique tangent plane to $B(x)$ at $\phi_1(x,y)$. Similarly, if 
$y\not=\{c,d\}$, then 
there is a unique support plane of $\conv(|\Sigma_1|\cup|\Sigma_2|)$ at $\phi_2(x,y)$ orthogonal to $\phi_1(x,y)\phi_2(x,y)$.

Note that $\phi_1(a,y)=a$ and $\phi_1(b,y)=b$ for every $y\in E(\Sigma_2)$, 
and similarly, $\phi_2(x,c)=c$ and $\phi_2(x,d)=d$ for every 
$x\in E(\Sigma_1).$  In any other case,
\[\phi_i\colon E(\Sigma_1)\times E(\Sigma_2)\to \mathbb{R}^3\text{ is an smooth embedding.}\]

Denote by $(AB)$ the image of $\phi_1\big(E(\Sigma_1)\times E(\Sigma_2)\big)$ and by $(CD)$ the image of 
$\phi_2\big(E(\Sigma_1)\times E(\Sigma_2)\big)$
and note that by the above, 
$(AB)$ and $(CD)$ are smooth surfaces with boundary,  smoothly  embedded in $\mathbb{R}^3$ and contained in the boundary of $|\Sigma_1|$ and  $|\Sigma_2|$, respectively.

On the other hand, observe that  $\tilde C_{AB}=
\{\phi_1(x,c)\mid x\in E(\Sigma_1)\}$ is a 
curve contained in the sphere of radius $r_1+r_2 -c_1c_2$ with center at $c$, connecting $a$ with $b$ (see Corollary \ref{cordevice}).  
Similarly, $\tilde D_{AB}=\{\phi_1(x,d)\mid x\in E(\Sigma_1)$\} is a curve contained in the sphere of radius $r_1+r_2 -c_1c_2$ with center at $d$, connecting $a$ with $b$.  
Thus,  the boundary of the surface $(AB)$ is 
\[\partial (AB)= \tilde C_{AB}\cup \tilde D_{AB}.\]

Likewise, $\tilde A_{CD}=
\{\phi_2(a,y)\mid y\in E(\Sigma_2)\}$ is a 
curve contained in the sphere of radius $r_1+r_2 -c_1c_2$ and center at $a$, connecting $c$ with $d$.
Similarly, $\tilde B_{CD}=\{\phi_2(b,y)\mid y\in E(\Sigma_2)\}$ is a curve contained in the sphere of radius $r_1+r_2 -c_1c_2$ with center at $b$, connecting $c$ with $b$.  Thus,  the boundary of the surface $(CD)$ is 
\[\partial (CD)= \tilde A_{CD}\cup \tilde B_{CD}.\]

The two surfaces with boundaries $(AB)$ and $(CD)$ are called the \emph{wedge-pod surfaces} of the confocal pair of pea pod devices $\Sigma_1$ and $\Sigma_2$.   Note that if we fix an $x=x_0$ while $y$ varies in $E(\Sigma_2)$, then $\phi_1(x_0,y)$ is an arc of circle contained in $(AB)$, while if we fix $y=y_0$ while $x$ varies in $E(\Sigma_1)$, then $\phi_2(x,y_0)$ is an arc of circle contained in $(CD)$.

In summary, we have the following lemma.

\begin{lema}\label{lemY}
Let $(AB)$ and $(CD)$  be the wedge-pod surfaces of the confocal pair of pea pod devices $\Sigma_1$ and $\Sigma_2$. 
Then $d\big((AB),(CD)\big)=r_1+r_2 -c_1c_2$. Moreover, 
for every point $u\in (AB)$, there is a point $v\in (CD)$ such that $uv=r_1+r_2 -c_1c_2$, and there is a   plane tangent to $(AB)$ at $u$ and a  plane tangent to $(CD)$ at $v$ orthogonal to  $uv$. Furthermore, if  $u\in (AB)\setminus \{a,b\}$, then the tangent plane of $(AB)$ at $u$ is unique, and similarly if $v\in (CD)\setminus \{c,d\}$, then the tangent plane of $(CD)$ at $v$ is unique. 
\end{lema}

Our strategy is now to construct the boundary of a body of constant width  by assembling wedge-pod surfaces together with sphere 
caps. See Figures \ref{patric} and \ref{topback}.

 \section{ Assembling  a Peabody of Constant Width}\label{assembling}

 Throughout the rest of the paper, we will assume that $\Sigma_1$ and $\Sigma_2$ are convex confocal pea pod devices with longitudinal beams $I_1$ and $I_2$.  Moreover, $I_1=ab$ and $I_2=cd$  are opposite sides of a regular tetrahedron $abcd$ of length, say $2$, and consequently by Corollary \ref{cordevice}, $r_1+r_2 -c_1c_2=2$.  Furthermore, by the convexity of the confocal pea pod devices, $E(\Sigma_i)\cap abcd=\{a,b\}$,  
$E(\Sigma_j)\cap abcd=\{c,d\}$, the orthogonal projections along $L$ of $E(\Sigma_1)$ and $ab$ coincide and the orthogonal projections along $L$ of $E(\Sigma_2)$ and $cd$ coincide.

 Next, for every other pair of opposite sides of the regular tetrahedron $abcd$, we will choose confocal pea pod devices, having them as longitudinal beams,  and as in the previous section  we construct six wedge-pod surfaces $(AB), (CD), (AC), (BD), (AD), (BC)$.   Each  of these surfaces has $\partial (AB)$, $\partial (CD)$, $\partial (AC)$, $\partial (BD)$, $\partial (AD)$ and $\partial (BC)$ as a boundary, with its corresponding pair of curves defined in the previous section.

   
Denote by  $\mathbb{S}(x,2)$ the sphere with center at $x$ and radius $2$. Observe that 
\[\tilde A_{BC}\cup \tilde A_{CD}\cup\tilde A_{DB}\subset \mathbb{S}(a,2)\] 
is a simple closed curve. 
Let $(A)$ be the spherical cap of $\mathbb{S}(a,2)$ bounded by the curve $\tilde A_{BC}\cup \tilde A_{CD}\cup\tilde A_{DB}$
in such a way that $(A)$ is contained in the boundary of the Reuleaux tetrahedron with vertices $\{a,b,c,d\}$. For the definition of the Reuleaux tetrahedron, see the paragraph preceding Lemma \ref{lemdia} and  \cite[Section 8.2]{MMO}).

In a similar way we define the sphere cap $(B)$ of   $\mathbb{S}(b,2)$ with boundary 
$\tilde B_{CD}\cup \tilde B_{DA}\cup\tilde B_{AC}$, the sphere cap  
$(C)$ of $\mathbb{S}(c,2)$ with boundary
$\tilde C_{AB}\cup \tilde C_{BD}\cup\tilde C_{DA}$, 
and  the sphere cap $(D)$ of $\mathbb{S}(d,2)$ with boundary 
$\tilde D_{AB}\cup \tilde D_{BC}\cup\tilde D_{CA}.$ 
  

Note that if $u\in \tilde A_{(BC)}\setminus\{b,c\}$, then the line through $au$ is a unique line normal to 
both surfaces, $(A)$ and $(BC)$. This allows us to glue together all these surface pieces; 
four sphere caps, one for each vertex, and six wedge-shaped surfaces, called wedge-pod surfaces, one for each side of the tetrahedron. In this way, we obtain a smooth submersion $\xi$ from the sphere  $\mathbb{S}^2$ into $\mathbb{R}^3$, 
\[\xi\colon \mathbb{S}^2\to \mathbb{R}^3\] 
with image $\xi(\mathbb{S}^2)=\Im$. (See  Figure \ref{topback}.)

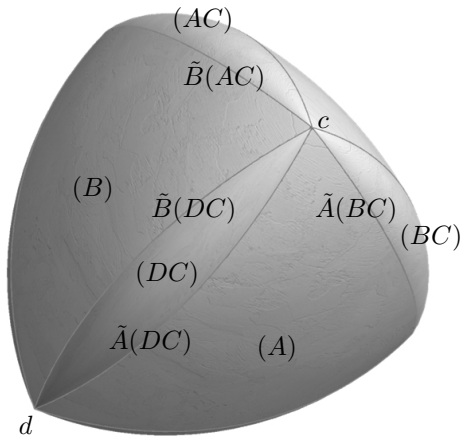
\begin{figure}[h]
\begin{center}
\includegraphics[scale=.5]{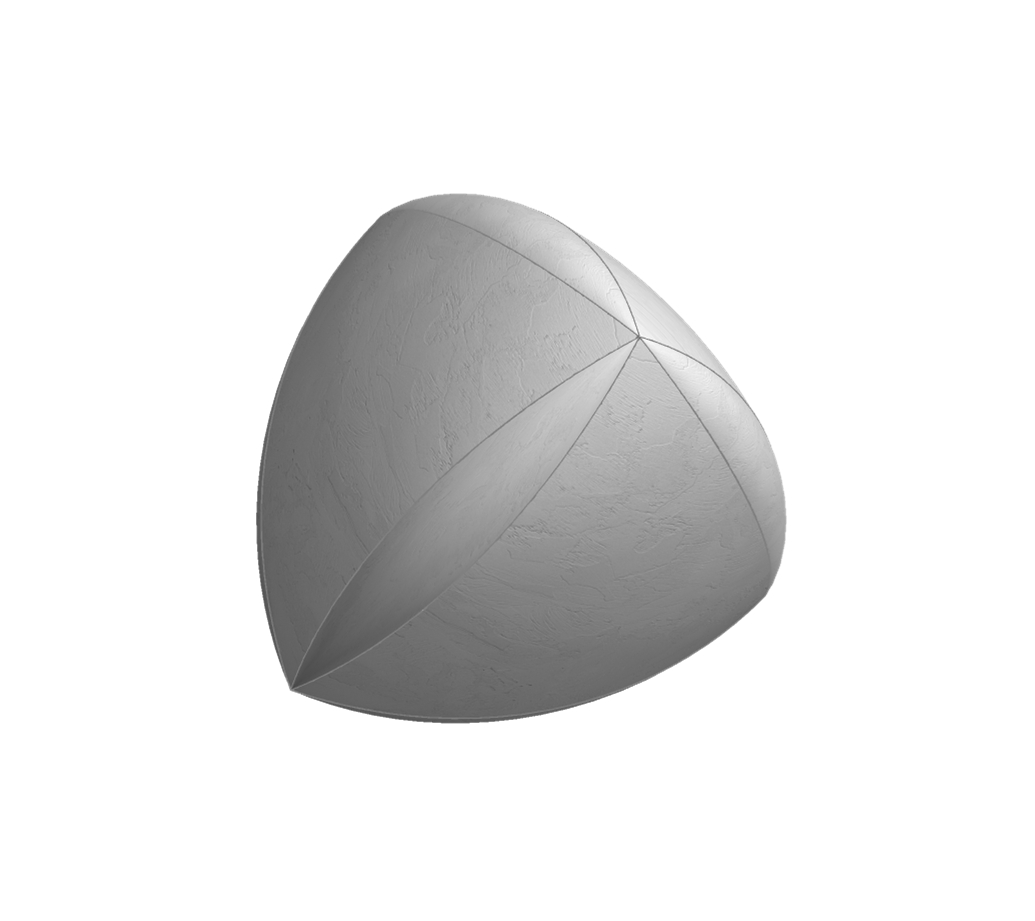}\put(-230,60){$d$}\put(-117,176){$c$}\put(-186,118){$(DC)$}\put(-140,90){$(A)$}\put(-210,150){$(B)$}\put(-172,214){$(AC)$}\put(-86,132){$(BC)$}\put(-180,142){$\tilde B(DC)$}\put(-196,92){$\tilde A(DC)$}\put(-168,192){$\tilde B(AC)$}\put(-118,142){$\tilde A(BC)$}
\caption{Assembling a peabody
}\label{topback}
\end{center}
\end{figure}

A consequence of Lemma \ref{lemY}, the following lemma  emphasizes the fundamental property of the smooth submersion, $\xi\colon \mathbb{S}^2\to \mathbb{R}^3$.

\begin{lema}\label{lemnormal}
For every point $x\in \mathbb{S}^2$,  there is a unique point $y\in \mathbb{S}^2$ such that the length of $\xi(x)\xi(y)$ is $2$ and 
the planes orthogonal to $\xi(x)\xi(y)$ at $\xi(x)$ and $\xi(y)$ are tangent planes of the smooth submersion $\Im$.  Furthermore, if  
$\xi(x)$ is different from $\{a,b,c,d\}$, the tangent plane of the submersion $\Im$ at $\xi(x)$ is unique. 
\end{lema}

\subsection{The Diameter of the Smooth Submersion  $\Im$}

In this section we would like to prove that $\Im$ is the boundary of a body of constant width. For this purpose, it is essential to prove that the diameter of $\Im$ is $2$. 

We begin by considering the following set of lines $L(x,y)$, corresponding to the opposite sides $ab$ and $cd$ of the tetrahedron $abcd$,
\[\mathcal{L}_{ab,cd}=\{L(x,y)\mid x\in E(\Sigma_1)\setminus \{a,b\} \text{ and }  y\in E(\Sigma_2)\setminus \{c,d\}\}.\]
Define $\mathcal{L}_{ac,bd}$ and $\mathcal{L}_{ad,bc}$ similarly, and let $\mathcal{L}=\mathcal{L}_{ab,cd}\cup \mathcal{L}_{ac,bd}\cup \mathcal{L}_{ad,bc}$.

\begin{lema}\label{lemfam}
The three families $\mathcal{L}_{AB,CD}, \mathcal{L}_{AC,BD},$ and  $\mathcal{L}_{AD,BC}$ are pairwise disjoint 
\end{lema}

\begin{proof}
Let $\mho$ be the space of all lines in $\mathbb{R}^3$ except for those lines parallel to a face of the tetrahedron $abcd$.
We shall prove that  $\mathcal{L}_{ab,cd}$, $\mathcal{L}_{ac,bd}$ and $\mathcal{L}_{ad,bc}$ lie in different components of $\mho$. 
The space $\mho$ consists precisely of those lines with the property that the orthogonal projection along them maps the vertices of the tetrahedron $abcd$ to four points in general position. Therefore, $\mho$ has seven connected components and three of them correspond to those lines with the property that the orthogonal projection along them maps the vertices of the tetrahedron $abcd$ to four points in convex position. Indeed, if  $L_1$ is the line through the midpoints of $ab$ and $cd$, $L_2$ is the line through the midpoints of $ac$ and $bd$ and $L_3$ is the line through the midpoints of $ad$ and $bd$, then one of $L_1$, $L_2$ and $L_3$ lies  in each of these three connected components of $\mho$.

We will next prove that $\mathcal{L}_{ab,cd}\subset \mho$. By symmetry, it is enough to verify that $L(x,y)\in \mathcal{L}_{ab,cd}$ is not parallel to the face $abc$ of the tetrahedron. Suppose $L(x,y)\subset \Gamma$, where $\Gamma$ is a plane parallel to the face $abc$. If the plane through $abc$ separates $\Gamma$ from $d$, then $\Gamma$ does not intersect $E(\Sigma_2)$, but if the plane through $abc$ does not separate $\Gamma$ from $d$, then $\Gamma$ does not intersect $E(\Sigma_1)$. This is so because of the convexity of the confocal pea pod devices $\Sigma_1$ and $\Sigma_2$. If $\Gamma$ is the plane through $abc$, then $L(x,y)\notin\mathcal{L}_{ab,cd}$, because $y=c$ and $x$ is equal to either $a$ or to $b$.
Similarly, $\mathcal{L}_{ac,bd}\subset \mho$ and $\mathcal{L}_{ad,bc}\subset \mho$.

The lemma follows now from the fact that these three sets $\mathcal{L}_{ab,cd}$, $\mathcal{L}_{ac,bd}$ and $\mathcal{L}_{ad,bc}$ are connected subsets of $\mho$, and each of $L_1\in \mathcal{L}_{ab,cd}$, $L_2\in \mathcal{L}_{ac,bd}$ and $L_3\in \mathcal{L}_{ac,bd}$  lies in a  different connected component of $\mho$.\end{proof}

Consider the tetrahedron $abcd$ with sides $2$ and having, at each of its vertices, a ball of radius $2$ centered at the vertex and containing the other three vertices on its boundary. It is well known that the intersection of these four balls of radius $2$ is the Reuleaux tetrahedron $T$.
It contains the tetrahedron $abcd$. and the singularities of the boundary of $T$ are an embedded copy of the complete graph  $K_4$, where the six edges are arcs of a circle and the vertices are $\{a,b,c,d\}$. Furthermore, the smooth pieces of the boundary of $T$ consist of four spherical caps of radius $2$. See Figure \ref{RT}. For more about the Reuleaux tetrahedron $T$, see \cite[Section 8.2]{MMO}.


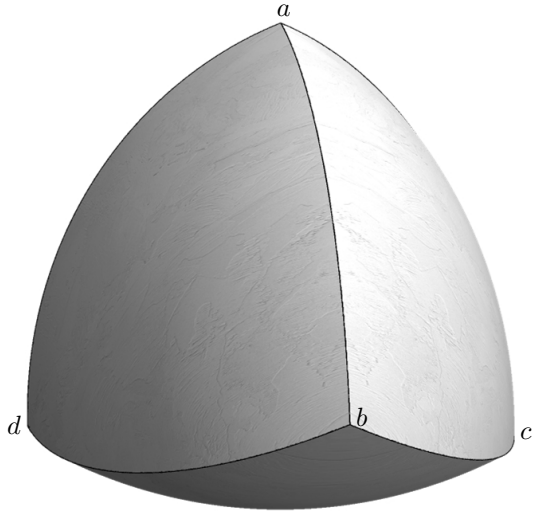
\begin{figure}[h]
\centering\includegraphics[scale=.5]{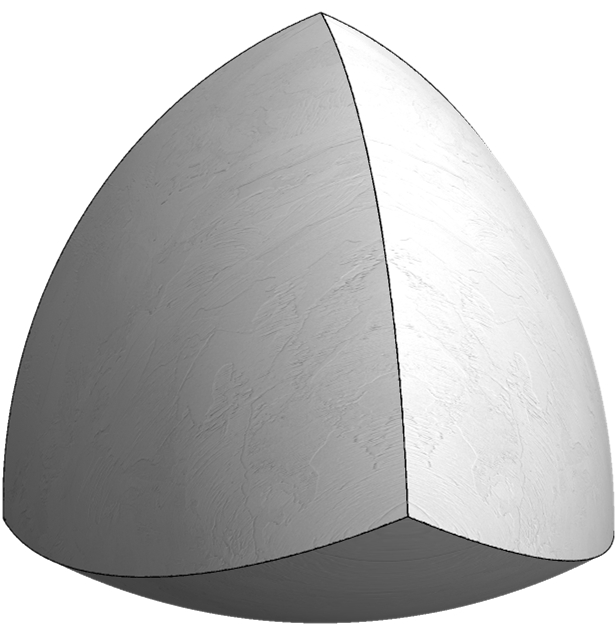}\put(-90,190){$a$}\put(-60,35){$b$}\put(2,30){$c$}\put(-192,32){$d$}
\caption{The Reuleaux tetrahedron}
 \label{RT}
 \end{figure}


\begin{lema}\label{lemdia}
$\Im$ is contained in the Reuleaux tetrahedron $T$. 
\end{lema}
\begin{proof} By symmetry, it will be enough to show that $(AB)$ is contained in the Reuleaux tetrahedron $T$. 
By the proof of Lemma \ref{lemS}, it is obvious that $(AB)$ is contained in the interior of $\mathbb{S}(c,2)$ and the interior of $\mathbb{S}(d,2)$.
Then it is left to show that $(AB)$ is contained in the interior of $\mathbb{S}(a,2)$ and in the interior of $\mathbb{S}(b,2)$ as well. 
Let $w$ be the point in $(AB)$ with the property that it is the furthest point away from $a$.  Suppose $w\in B(x)$ for some $x\in E(\Sigma_1)$. Then the unique line normal to $(AB)$ at $w$ must contain $a$ and $x$, the center of $B(x)$, but then 
$B(x)$ is the ball of radius zero centered at $b$ and hence $w=b$. This implies that $(AB)$ is contained in the interior of $\mathbb{S}(a,2)$. Simillarily, $(AB)$ is contained in the interior of $\mathbb{S}(b,2)$. \end{proof}


\begin{lema}\label{lemdia2}
The diameter of $\Im$ is $2$.
\end{lema} 

\begin{proof} Let $xy$ be a diameter of $\Im$ and let $L$ be the line through $x$ and $y$. If $x$ lies in some of the sphere caps, say $x\in(A)$, then the vertex $a$ lies in $L$. Therefore, $xa$ is a chord of the Reauleaux tetrahedron, but by Lemma \ref{lemdia}, the Reauleaux tetrahedron contains $\Im$ and consequently $xa$ is a chord of $\Im$. This implies that $y=a$ and therefore, in this case, the diameter of $\Im$ is $2$.

Suppose now that $x$ lies in the relative interior of one of the wedge-pod surfaces, say, $(AB)$, and hence $L=L(x_0,y_0)$, for some $x_0\in E(\Sigma_1)\setminus\{a,b\}$ and $x=\phi_1(x_0,y_0)$.  Similarly $y$ lies in the interior of another wedge-pod surface and $L$, its normal line at $y$, must be another element of $\mathcal{L}$. By Lemma \ref{lemfam}, $L=L(x',y')$ for some $y'\in E(\Sigma_2)\setminus \{c,d\}$  
and $y=\phi_2(x',y')$.  Now, it is  easy to verify that $L=L(x_0,y_0)=L(x',y')$ implies that $x=x'$ and $y=y'$. Consequently  $y=\phi_2(x_0,y_0)$ and hence the length of the diameter $xy$ is $2$, as we wished. 
\end{proof}

\begin{teo}\label{teomain}
The surface $\Im$ is the boundary of a body of constant width $2$.
\end{teo}
\begin{proof} Consider the convex hull $\conv(\Im)$ of $\Im$. By Lemma \ref{lemdia2}, the diameter of $\conv(\Im)$ is $2$. Moreover, by 
Lemma \ref{lemnormal}, every point of $\Im$
is the extreme point of a diameter of $\conv(\Im)$. This implies immediately that $\Im$ is the boundary of $\conv(\Im)$. Moreover, again by 
Lemma \ref{lemnormal}, 
$\Im$ is smooth with the exception of the $4$ vertices of the tetrahedron $a,b,c,d$ at which $\Im$ has vertex singularities. Finally, \cite[Theorem 3.1.7]{MMO} implies that $\conv(\Im)$ is a body of constant width. \end{proof}

\section{Parabolic Devices and Robert's Body of Constant Width}

 In this section, we will study the case of parabolic pea pod devices; that is, when the frames of the two pea pod devices satisfy that one of the circles is a line through $I$ and the confocal quadrics of the pea pod devices are both parabolas. Furthermore, the spheres of the pea pods rest over the longitudinal beams.
 
 We begin by giving a couple of definitions. Two sets $S, T\subset \mathbb{R}^3$ are \emph{similar} if there is a homothecy $h$ and an isometry $\rho$ such that $\rho h(S)=T$.
A tetrahedron with vertices at $\{a,b,c,d\}$ is called \emph{semi-regular} if the line through the midpoints of $ab$ and $cd$ is orthogonal to both $ab$ and $cd$.

\begin{lema}\label{lemsimilar}
Given  a semi-regular tetrahedron with vertices $\{a,b,c,d\}$ and a pair of confocal quadrics $C_1,C_2$, there is a pair of confocal convex pea pod devices $\Sigma_1$, $\Sigma_2$ with longitudinal beams $ab$ and $cd$ and with convex confocal quadric pea strings similar to $C_1,C_2$.    
\end{lema} 

\begin{proof} Let $C_1, C_2$ be a pair of confocal quadrics. Suppose without loss of generality that $C_1$ lies in the $xz$-plane and $
C_2$ in the $yz$-plane, while the foci lie on the $z$-axis. See Figure \ref{confo}. Suppose $a'b'$ is a chord of $C_1$; that is, $a',b'\in C_1$, and $a'b'$ is contained in the $xz$-plane and furthermore, assume that $a'b'$ is orthogonal to the $z$-axis. 
Consider the semi-regular tetrahedron $a'b'c'd'$ similar to $abcd$. Then $c'd'$ is orthogonal to the $z$-axis and is contained in the $yz$-plane. 
If $a'b'$ is sufficiently small; that is, if $a'b'$ is close to the focus of $C_1$, then the line through $c'd'$ intersects $C_2$ in an interval $c''d''$ in such a way that $c'd'\subset c''d''$. On the other hand, if $a'b'$ is sufficiently close to the focus of $C_2$, then $c''d''\subset c'd'$. By continuity, there is a position of $a'b'$ for which $c''d''=c'd'$, as we wished. \end{proof}

Since up to similarity there is only one pair of confocal parabolas, we have the following corollary.

\begin{cor}\label{corsemi}
Given  a semi-regular tetrahedron with vertices $\{a,b,c,d\}$, there is only one pair of convex confocal pea pod devices 
$\Sigma_1$, $\Sigma_2$ with longitudinal beams $ab$ and $cd$ and with  confocal  parabolic pea strings.    
\end{cor}

\subsection {Robert's Body with Tetrahedral Symmetry}

The body of constant width obtained from the regular tetrahedron $abcd$ with sides of length $2$ by always using  pea pod devices  with  confocal  parabolic pea strings will be called  the \emph{Robert's body} $\mathcal{R}$ of constant width.  This body was informally described by Patrick Roberts at the web page \cite{RP}. 
Observe that Robert's body can not be the extreme body of the Blaschke-Lebesgue conjecture because Anciaux and Guilfoyle  proved in \cite{AG} that the body that minimizes the volume among all 3-dimensional bodies of constant width has the property that the smooth parts of its boundary are spherical caps or surfaces of rotation, which is not the case for Robert's body. See \cite[Section 14.2]{MMO}. The boundary of Robert's body is smooth with the exception of the $4$ vertices of the tetrahedron, at which the boundary has vertex singularities.

Next we will show that Robert's body has the symmetry of the tetrahedron $abcd$. Using the unique pair of confocal pea pod devices $\Sigma_1$, $\Sigma_2$ with longitudinal beams $ab$ and $cd$ and with confocal  parabolic pea strings, we obtain the wedge-pod surfaces $(AB)$ and $(CD)$, defined as in Section \ref{assembling}.

Let $L$ be the axis of the devices; that is, $L$ is the line through the midpoints of $ab$ and $cd$, and suppose the origin $O$ is the barycenter of the tetrahedron $abcd$. Let 
$\rho_{\pi/2}$ and $\rho_{\pi}$  be the rotations along $L$ by an angle of $\pi/2$ and $\pi$ respectively, and let $\tau_1$ be the reflection through the plane of $\Sigma_1$ and  $\tau_2$ the reflection through the plane of $\Sigma_2$,  
 and let $\alpha$ be the antipodal map. Since the 
 confocal   parabolic pea strings are invariant under $\rho_{\pi}$, $\alpha\rho_{\pi/2}$, $\tau_1$ and $\tau_2$,  by construction, the same holds for the wedge-pod surfaces $(AB)\cup(CD)$. Now, using the unique pair of confocal pea pod devices with longitudinal beams $ac$ and $bd$ with confocal  parabolic pea strings, we obtain the corresponding  wedge-pod surfaces $(AC)$ and $(BD)$ and in similar fashion the corresponding wedge-pod surfaces $(AD)$ and $(BC)$.  By Corollary \ref{corsemi}, there is an isometry that sends $(AB)$ to $(AC)$ and $(CD)$ to $(BD)$.  All this implies that Robert's body has tetrahedral symmetry.

 \subsection{Robert's Body and the Minkowski Sum of the Meissner Bodies}
 
 Although Robert's body and the Minkowski sum of the two well known Meissner bodies of constant width (see \cite[Section 8]{MMO}) both have constant width $2$, both   contain the tetrahedron $abcd$, both have  tetrahedral symmetry, and  the boundary of both is smooth with the exception of the $4$ vertices of the tetrahedron where they have vertex singularities, they are not the same body. We will observe next that they are essentially different, since they differ in one of their sections, as the following theorem shows.

 \begin{teo}
 The Minkowski sum of the two Meissner bodies of constant width $1$ is not the Robert's body $\mathcal{R}$.
 \end{teo}

 \begin{proof} Let $M_1$ and $M_2$ be the two Meissner bodies of constant width $1$ constructed over the regular tetrahedron 
$\{\frac{a}{2},\frac{b}{2},\frac{c}{2},\frac{d}{2}\}$, and let $M=M_1+M_2$. 
Let $H$ be the plane that contains the side $ab$ and passes through the midpoint of the side $cd$. Note that in both cases, reflection about $H$ leaves $M$ invariant. This implies that $H\cap M$ has constant width and hence that 
 $H\cap M=(H\cap M_1)+(H\cap M_2)$. On the other hand, note that $(H\cap M_1)$ is a Reuleaux triangle whereas 
 $(H\cap M_2)$ is the figure of constant width depicted in Figure \ref{isaac2} for some parameters $r_1,r_2>0$ and $r_1+r_2=1$. Let us call this figure the $r_1$-figure of constant width $1$. The Reuleaux triangle is thus the $1$-figure of constant width $1$. On the other hand, the section 
 $H\cap\mathcal{R}$ shown in the second picture of Figure \ref{isaac2} is also an $r$-figure for the parameter $r>0$, where $r$ is the radius of the principal circle of the parabolic pea pod device.
 It is clear now that the Minkowski sum of the Reuleaux triangle with the $r_2$-figure, $r_2>0$, is never an $r$-figure of constant 
 width $1$. \end{proof}

 \begin{figure}[h]
\begin{center}
\includegraphics[scale=.34]{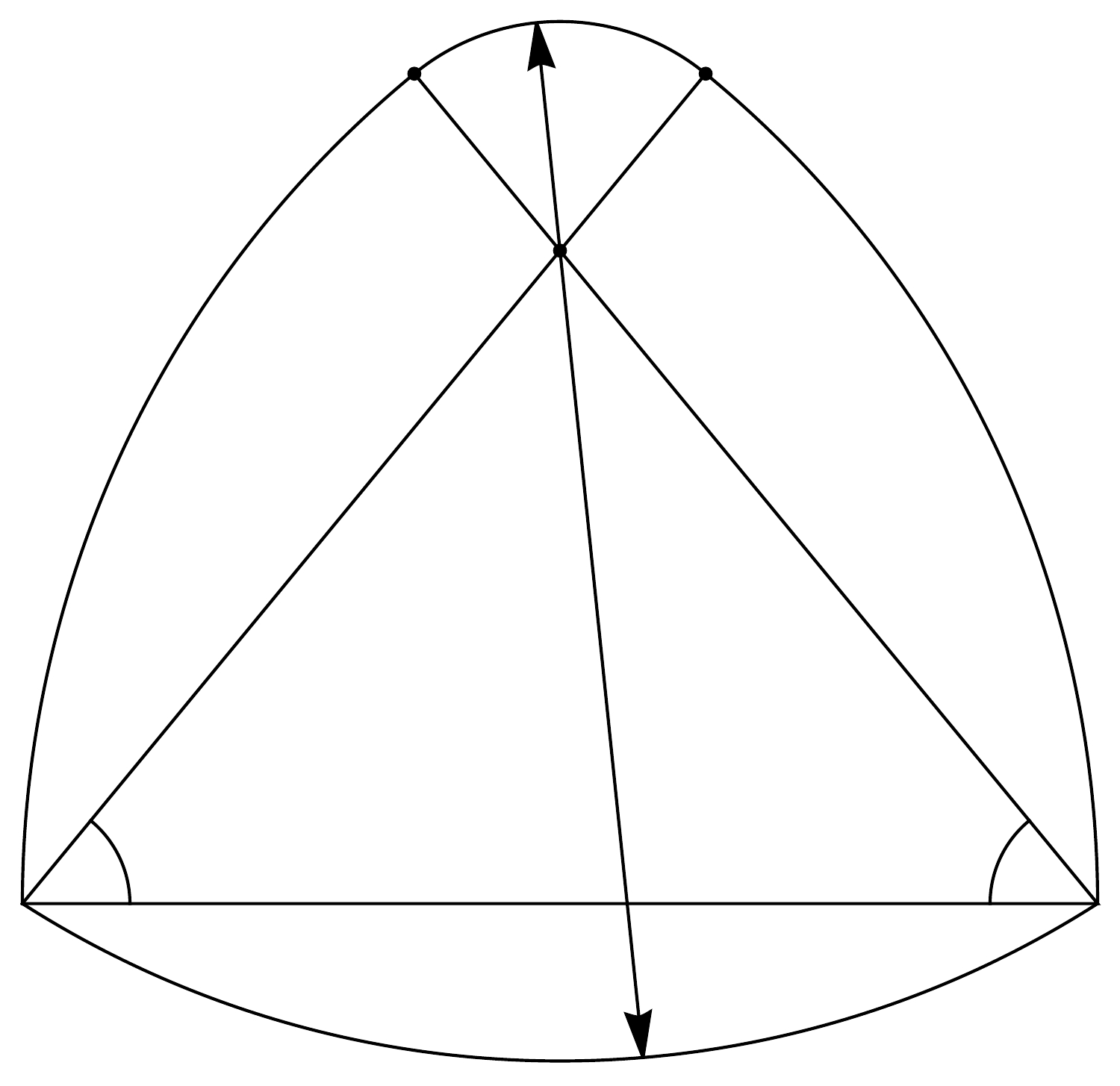}\put(-128,17){$a$}\put(0,17){$b$}\put(-66,45){$r_1$}\put(-62,104){$r_2$}\ \ \ \ \ \ 
\includegraphics[scale=.34]{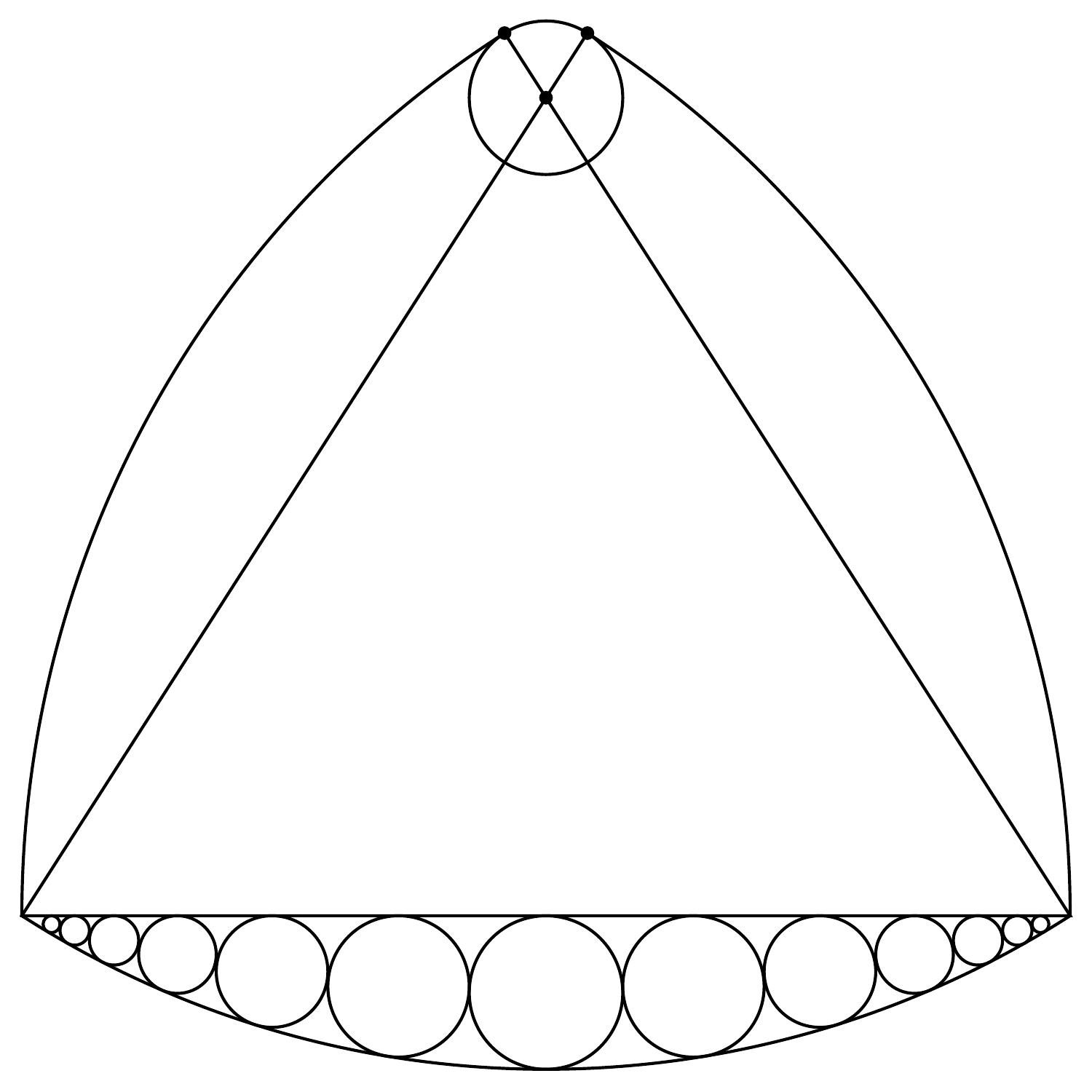}\put(-128,17){$a$}\put(0,17){$b$}
\caption{The $r_2$-figure and a section of Robert's body }\label{isaac2}
\end{center}
\end{figure}
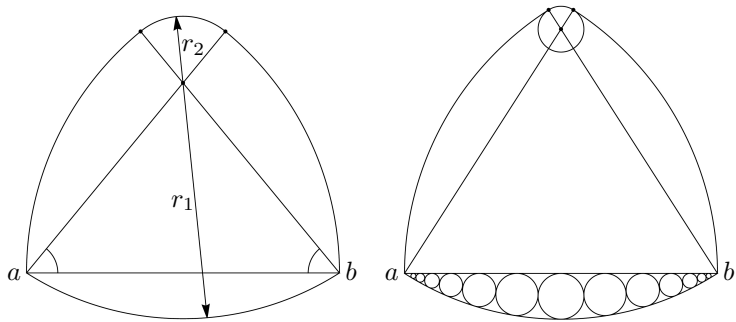

 \subsection{Classic Meissner Bodies as Peabodies}

Take $C_e$ and $C_h$, a confocal ellipse and hyperbola respectively, and let the focus of the ellipse come closer and closer to the center.
 At the limit situation, $C_e$ becomes a circle and $C_h$ a line orthogonal to the circle through its center. This pair of curves may be considered as confocal quadrics in our construction.   

Let $abcd$ be the regular tetrahedron of side $2$. Using the unique pair of confocal pea pod devices $\Sigma_1$, $\Sigma_2$ with longitudinal beams $ab$ and $cd$ and with  confocal  pea strings, the circle and the orthogonal line,  we obtain the wedge-pod surfaces $(AB)$ and $(CD)$ as in Section \ref{wedges}.  We can see that  the two circles of $\Sigma_1$ coincide and the center is the midpoint of $cd$. So, in this case, the collection of disks of the pea pod of $\Sigma_1$ consists of points, and hence the corresponding wedge-pod surface $(AB)$ is, in this degenerate case, the arc of a circle with its center at the midpoint of $cd$ from $a$ to $b$. On the other hand, the hyperbolic pea pod device $\Sigma_2$ is such that its quadric pea string is the line through $cd$ and therefore all the disks of the pea pod of $\Sigma_2$ have their center at $cd$. This implies that the wedge-pod surface $(CD)$ is  a surface of revolution along $cd$. As can be seen, this is precisely the surgery procedure  given in the construction of  the Meissner bodies. See \cite[Section 8]{MMO}. 
Finally, observe that there is a continuous deformation along the collection of peabodies of constant width from the most symmetric such body, the Robert's body, to the Meissner bodies. This is because, by Lemma \ref{lemsimilar}, we can use any pair of confocal quadrics to construct the pod surfaces $(AB)$ and $(CD)$.

\section{General Meissner Peabody Polyhedra}

We will finish this paper by extending this construction to a more general one. In \cite{MR}, Montejano and Roldan used metric embeddings of self-dual graphs to construct bodies of constant width, called Meissner polyhedra.
Let $G$ be a metric embedding of a self-dual graph and let $ab$ and $cd$ be a pair of dual edges of $G$. It is easy to see that $\{a,b\}$ and $\{c,d\}$ are the vertices of a semi-regular tetrahedron.  Using Corollary \ref{corsemi},  we can see that there is only one pair of convex confocal pea pod devices $\Sigma_1$, $\Sigma_2$ with longitudinal beams $ab$ and $cd$ and with  confocal  parabolic pea strings. This allows us to obtain the wedge-pod surfaces $(AB)$ and $(CD)$ and  perform  surgery along two dual singularity edges of the ball polyhedra $B(G)=\bigcup_{v\in V(G)} B(v)$.
By performing this procedure for any pair of dual edges of $G$, we can assemble, exactly as in Section \ref{assembling}, a peabody of constant width $\mathcal{R}_G$ by replacing a small neighborhood of all edges of $G$ with sections of an envelope of spheres as before. Note that by the discussion in this last two sections, every symmetry of $G$ is also a symmetry of $\mathcal{R}_G$.

\subsection*{Acknowledgments} Luis Montejano and Deborah Oliveros acknowledge  support  from CONACyT under 
project CONACyT 282280 and  from PAPIIT-UNAM under project IG100721.

\end{document}